\documentclass[12pt,regno]{amsart}
\usepackage{epsf,epsfig,amsmath}
\usepackage{amssymb,latexsym}
\usepackage{tikz}
\usepackage{amsthm} 
\usepackage{amsfonts}
\usepackage{graphicx,psfrag}

\numberwithin{equation}{section}

\newtheorem{thm}{Theorem}[section]
\newtheorem*{thm*}{Theorem}
\newtheorem{pro}[thm]{Proposition}
\newtheorem{lem}[thm]{Lemma}

\newtheorem{cor}[thm]{Corollary}

\theoremstyle{definition}
\newtheorem{defi}[thm]{Definition}
\newtheorem{exa}[thm]{Example}
\theoremstyle{remark}

\newtheorem{rem}[thm]{Remark}
\title[A characterization of special matchings]{A simple characterization of special matchings in lower Bruhat intervals}

\author{Fabrizio Caselli}\author{Mario Marietti}

\address{Fabrizio Caselli, Dipartimento di matematica, Universit\`a di Bologna, Piazza di Porta San Donato 5, 40126 Bologna, Italy}
\address{Mario Marietti, Dipartimento  di Ingegneria Industriale e Scienze Matematiche, Universit\`a Politecnica delle Marche, Via Brecce Bianche, 60131 Ancona,  Italy}

\email{fabrizio.caselli@unibo.it}
\email{m.marietti@univpm.it}

\subjclass[2010]{05E99 (primary), 20F55 (secondary)}
\keywords{Bruhat order, Coxeter group, Special matching}

\begin{document}

\maketitle

\begin{abstract}
We give  a simple characterization of special matchings in lower Bruhat intervals (that is, intervals starting from the identity element) of a Coxeter group. As a byproduct, we obtain some results on the action of special matchings. 
\end{abstract}

\section{Introduction}

Over the last few years, special matchings have shown many applications in Coxeter group theory in general (see \cite{MJaco}, \cite{MSiam}) and in the computation of  Kazhdan--Lusztig polynomials in particular (see \cite{Brepr}, \cite{BCM1}, \cite{BCM2}, \cite{Mtrans}, \cite{Mprepr}, \cite{Tel}).

The main achievements obtained by using special matchings are in the problem of the combinatorial invariance of Kazhdan--Lusztig polynomials. These famous polynomials, introduced by Kazhdan and Lusztig in \cite{K-L}, are polynomials indexed by pairs of elements $u,v$ in a Coxeter group, with $u\leq v$ under Bruhat order. In the 80's, Lusztig in private and, independently,  Dyer \cite{Dyeth} have conjectured that the Kazhdan--Lusztig polynomial $P_{u,v}(q)$ only depends on the combinatorial structure of the interval $[u,v]$ (that is, on the isomorphism class of $[u,v]$ as a poset). This conjecture is often referred to as the Combinatorial Invariance Conjecture of Kazhdan--Lusztig polynomials. 

Over the last 15 years, new results about the Combinatorial Invariance Conjecture (and its generalization to the parabolic setting) for lower intervals, that is, intervals starting from the identity element, were obtained by proving a recursive formula for Kazhdan--Lusztig polynomials which depends on special matchings. This result was first obtained for the ordinary Kazhdan--Lusztig polynomials of the symmetric group \cite{Brepr}, then for the ordinary Kazhdan--Lusztig polynomials of an arbitrary Coxeter group \cite{BCM1} (see also \cite{Del}), then for the parabolic  Kazhdan--Lusztig polynomials of a doubly-laced Coxeter group \cite{Mtrans}, and very recently for the parabolic  Kazhdan--Lusztig polynomials of an arbitrary Coxeter group \cite{Mprepr}. Each time, these improvements were made possible by a better control on the special matchings of Coxeter groups. The impression is that a further understanding of special matchings of Coxeter groups might bring to other results on the combinatorial invariance of Kazhdan--Lusztig polynomials.

The purpose of this paper is to give a simple characterization of special matchings of lower intervals in any arbitrary Coxeter group.

With respect to the classification given in \cite{CM}, the present classification has the advantages of being more compact and simpler. It is also aesthetically pleasing to have only one self-dual type of system instead of two types of systems, one dual to the other. Furthermore, we provide a simpler  description of the action of a special matching, which is easy to work with. More precisely, given a special matching $M$ of an element $w\in W$, we  show that, for all $u \in W$ with $u\leq w$, the element $M(u)$ can be computed from any factorization of $u$ satisfying certain hypotheses, while in \cite{CM} we could compute $M(u)$ only starting from  one specific factorization.

As an application of the classification in this paper, we give a new proof of the Combinatorial Invariance Conjecture for lower Bruhat intervals in any Coxeter group, which is shorter then the original one in \cite{BCM1}.

\section{Statements of main results}\label{sor}
We fix an arbitrary Coxeter system $(W,S)$. 

For $w\in W$,  the \emph{support} of $w$, denoted  $\textrm{Supp}(w)$, is $\{s\in S \colon\,  s\leq w\}$. For $H\subseteq S$, we let $\textrm{Supp}_H(w)$ denote the intersection $\textrm{Supp}(w)\cap H$. For $s\in S$, we denote by $C_s$ the set of generators commuting with $s$, that is $\{c\in S \colon\, sc=cs\}$. For $w\in W$ and $H\subseteq S$, we denote by $w_0(H)$ the longest element of $[e,w]\cap W_H$.
\begin{defi}
Let $s\in S$ and $J,K\subseteq S$. We say that $J$ and $K$ are \emph{$s$-complementary} if $J\cup K=S$ and $J\cap K=C_s$. We also say that $K$ is the \emph{$s$-complement} of $J$, and vice versa.
\end{defi}

\begin{defi}
\label{sistema}
A \emph{system} $\mathcal S$ for $w\in W$ is a triple $(J,H,M)$ with $H\subseteq S$, $|H| \in \{1,2\}$, $M$  a matching of the Hasse diagram of $[e,w_0(H)]$, and $C_{M(e)}\subseteq J\subseteq S$, such that
 
\begin{itemize}
 \item [[S0]] $M$ is a multiplication matching (if and) only if $|H|=1$;
 \item [[S1]] $w^J\in W_K$ (equivalently $^Kw\in W_J$), where $K$ is the $M(e)$-complement of $J$;
 \item [[S2]]  
 \begin{itemize}
 \item $|\textrm{Supp}_H((w^J)^H)|\leq 1$ and if $\alpha\in \textrm{Supp}_H((w^J)^H)$ then $M$ commutes with $\lambda_\alpha$,
 \item $|\textrm{Supp}_H(^H(^Kw))|\leq 1$ and if $\beta \in  \textrm{Supp}_H(^H(^Kw))$ then $M$ commutes with $\rho_\beta$.
\end{itemize}
\end{itemize}
\end{defi}

\begin{defi} 
 Let $\mathcal S$ be a system $(J,H,M)$ for $w$, and $u\in W$. An \emph{$\mathcal S$-factorization} of $u$ is a factorization of $u=a \cdot b  \cdot c$, for some $a,b,c\in W$ satisfying the following properties:
 \begin{itemize}
  \item $\ell(u)=\ell(a)+\ell(b)+\ell(c)$;
  \item $a \in W_{K}\cap W^H$, $|\textrm{Supp}_H(a)|\leq 1$, and if $\alpha \in \textrm{Supp}_H(a)$ then $M$ commutes with $\lambda_{\alpha}$;
  \item $b\in W_H$;
  \item $c\in W_J \cap  \,^HW$, $|\textrm{Supp}_H(c)|\leq 1$, and if $\beta \in \textrm{Supp}_H(c)$ then $M$ commutes with $\rho_{\beta}$.
\end{itemize}
\end{defi}

In the sequel, we prove that, to each system $\mathcal S=(J,H,M)$, we can attach a special matching $M_{\mathcal S}$ of $w$ by letting, for all $u\in W$ with $u\leq w$
$$M_{\mathcal S}(u)= a \cdot M(b)  \cdot c$$
where $u=a\cdot b \cdot c$ is an arbitrary $\mathcal S$-factorization. While it is easy to see that $\mathcal S$-factorizations exist for all $u\in W$ with $u\leq w$, the fact that $M_{\mathcal S}$ does not depend on the chosen $\mathcal S$-factorization and is indeed a matching of $w$ (when $M_{\mathcal S}(w)\lhd w$) is a key point in the sequel. 

The main results of this work are collected in the following.
\begin{thm*}
\begin{itemize}
Let $(W,S)$ be an arbitrary Coxeter system and $w$ be an arbitrary element of $ W$.
\item
If $M$ is a special matching of $w$, then there exists a system $\mathcal S$ for $w$ such that $M=M_{\mathcal S}$. 
\item  
 Vice versa, if $\mathcal S$ is a system for $w$ such that $M_\mathcal S(w)\lhd w$, then $M_\mathcal S$ is a special matching for $w$.
 \item 
If $M$ is a special matching of $w$ associated with  a system $\mathcal S$, $w=a\cdot b\cdot c$ is an $\mathcal S$-factorization of $w$, and $u$ is an element  smaller than $w$, then 
$$M(u)=a' M_\mathcal S(b') c',$$
for all factorizations $u=a' \cdot b' \cdot c'$ of  $u$ such that $a'\leq a$, $b'\leq b$, $c'\leq c$, and $\ell(u)=\ell(a')+\ell(b')+\ell(c')$.

\item  
Moreover,  let $SM_w$ be the set of systems $\mathcal S=(J,H,M)$ for $w$ such that
\begin{itemize}
\item  $M_\mathcal S(w)\lhd w$, 
\item for all $r\in S$,  
$$ M_{\mathcal S}(r)=sr \text{ if and only if } r\in J.$$
\end{itemize}
  Then
the set 
 \[
 \{ M_{\mathcal S}: \mathcal S\in SM_w\}
 \]
is a complete list of all distinct special matchings of $w$.
\end{itemize}
\end{thm*}

\section{Notation, definitions and background}

In this section, we collect  some notation, definitions,
 and results that are used in the rest of this work.

We  follow \cite[Chapter 3]{StaEC1} for undefined notation and 
terminology concerning partially ordered sets. In particular,
 given $x$ and $y$ in a partially ordered set  $P$, we say that $y$ {\em covers} $x$ and we write $x \lhd y$ if the interval $ \{ z \in P\colon\, x \leq z \leq y \}$, denoted $[x,y]$, has two elements,  $x$ and $y$. We say that a poset $P$ is {\em graded} if $P$ has a minimum and there is a function
$\rho : P \rightarrow \mathbb N$ such that $\rho (\hat{0})=0$ and $\rho (y) =\rho (x)
+1$ for all $x,y \in P$ with $x \lhd y$. 
(This definition is slightly different from the one given in \cite{StaEC1}, but is
more convenient for our purposes.) 
We  call $\rho $ the {\em rank function}
of $P$.
The {\em Hasse diagram} of $P$ is any drawing of the graph having $P$ as vertex set and $ \{ \{ x,y \} \in \binom {P}{2} \colon\,  \text{ either $x \lhd y$ or $y \lhd x$} \}$ as edge set, with the convention that, if $x \lhd y$, then the edge $\{x,y\}$ goes upward from $x$ to $y$. When no confusion arises, we make no distinction between the Hasse diagram and its underlying graph.

A {\em matching} of a poset \( P \) is an involution
\( M \colon\, P\rightarrow P \) such that \( \{v,M(v)\}\) is an edge in the Hasse diagram of $P$, for all \( v\in P \).
A matching \( M \) of \( P \) is {\em special} if\[
u\lhd v\Longrightarrow M(u)\leq M(v),\]
 for all \( u,v\in P \) such that \( M(u)\neq v \).

The following result (see \cite[Lemma~2.6]{MJaco}) is an immediate generalization  of the well-known Lifting Property for Coxeter groups (see, for example, \cite[Proposition~2.2.7]{BB} or \cite[Proposition~5.9]{Hum}).
\begin{lem}
\label{lpfsm}
 (Lifting Property for special matchings) Let $M$ be a special matching of
a locally finite ranked poset $P$, and let $u, v \in P$, with $u \leq  v$. Then
\begin{enumerate}
\item if $M(v) \lhd v$ and $M(u) \lhd u$ then $M(u) \leq M(v)$,
\item if  $M(v) \rhd v$ and $M(u) \rhd u$ then $M(u) \leq M(v)$,
\item if $M(v) \lhd v$ and $M(u) \rhd u$ then $M(u) \leq v$ and $u \leq M(v)$.
\end{enumerate}
\end{lem}

We follow \cite{BB} for undefined Coxeter groups notation
and terminology. 

Given a Coxeter system $(W,S)$ and $s,r\in S$, we denote by $m_{s,r}$ the order of the product $sr$. Given $w \in W$, we denote by $\ell
(w )$ the length of $w $ with respect to $S$, and by $  D_{R}(w )$ and $  D_{L}(w )$ the sets  $ \{ s \in S \colon\, 
\ell(w s) < \ell(w ) \}$ and 
$ \{ s \in S \colon\,  \ell(sw)<\ell(w)\}$, respectively. We call the elements of $ D_R(w )$ and $D_{L}(w )$, respectively, the {\em right descents} and the 
{\em left descents} of $w $.
We denote by $e$
the identity of $W$, and we let $T =\{ w s w ^{-1} \colon\,  w \in W, \; s \in S \}$ be the set of {\em reflections} of $W$.

The Coxeter group $W$ is partially ordered by {\em Bruhat order} (see, for example,  \cite[\S 2.1]{BB} or \cite[\S 5.9]{Hum}), which is denoted by $\leq$. The Bruhat order is the partial order whose  covering relation $\lhd$ is as follows: if $u,v \in W$, then $u \lhd v$ if and only if $u^{-1}v \in T$ and $\ell (u)=\ell (v)-1$. 
There is a well-known characterization of Bruhat order
on a Coxeter group (usually referred to as the {\em Subword Property})
that we  use repeatedly in this work,
often without explicit mention. We recall it here for the reader's
convenience. 

By a {\em subword} of a word $s_{1}\textrm{-}s_{2} \textrm{-} \cdots \textrm{-} s_{q}$ (where we use the symbol ``$\textrm{-}$'' to separate letters in a word in the alphabet $S$) we mean
a word of the form
$s_{i_{1}}\textrm{-} s_{i_{2}}\textrm{-} \cdots \textrm{-} s_{i_{k}}$, where $1 \leq i_{1}< \cdots
< i_{k} \leq q$.
If $w\in W$ then a {\em reduced expression} for $w$ is a word $s_1\textrm{-}s_2\textrm{-}\cdots \textrm{-}s_q$ such that $w=s_1s_2\cdots s_q$ and $\ell(w)=q$. When no confusion arises, we also say in this case that $s_1s_2\cdots s_q$ is a reduced expression for $w$. 
\begin{thm}[Subword Property]
\label{subword}
Let $u,w \in W$. Then the following are equivalent:
\begin{itemize}
\item $u \leq w$ in the Bruhat order,
\item  every reduced expression for $w$ has a subword that is 
a reduced expression for $u$,
\item there exists a  reduced expression for $w$ having a subword that is 
a reduced expression for $u$.
\end{itemize}
\end{thm}
A proof of the preceding result can be found in \cite[\S 2.2]{BB} or \cite[\S 5.10]{Hum}. 
It is well known that $W$, partially ordered by Bruhat order, is a graded poset having $\ell$ as its rank function.

We recall that two reduced expressions of an element are always linked by a sequence of {\em braid moves}, where a braid move consists in substituting a factor $s\textrm{-}t\textrm{-}s\textrm{-} \cdots$ ($m_{s,t}$ letters) with a factor $t\textrm{-}s\textrm{-}t\textrm{-} \cdots$ ($m_{s,t}$ letters), for some $s,t\in S$. We also recall that, if $w\in W$ and $s,t\in D_R(w)$, then there exists a reduced expression for $w$ of the form $s_1\textrm{-}\cdots \textrm{-} s_k\textrm{-}\underbrace{s\textrm{-}t\textrm{-}s\textrm{-} \cdots}_{m_{s,t}\textrm { letters}}$.

Given $u,v\in W$, we write $u\cdot v$ instead of simply $uv$ when $\ell(uv)=\ell(u)+\ell(v)$ and we want to stress this additivity. On the other hand, when we write $uv$, we mean that $\ell(uv)$ can be either $ \ell(u)+\ell(v)$ or smaller.

For each subset $J\subseteq S$, we denote by  $W_J $ the parabolic subgroup of $W$ generated by $J$, and by  $W^{J}$ the set of minimal coset representatives:
$$W^{J} =\{ w \in W \colon\,  D_{R}(w)\subseteq S\setminus J \}.$$
The following is a useful factorization of $W$ (see, for example,  \cite[\S 2.4]{BB} or \cite[\S 1.10]{Hum}).
\begin{pro}
\label{fattorizzo}
If  $J \subseteq S$, then
\begin{enumerate}
\item[(i)] 
every $w \in W$ has a unique factorization $w=w^{J}  w_{J}$ 
with $w^{J} \in W^{J}$ and $w_{J} \in W_{J}$;
\item[(ii)] for this factorization, $\ell(w)=\ell(w^{J})+\ell(w_{J})$.
\end{enumerate}
\end{pro}
There are, of course, left versions of the above definition and
result. Namely, if we let  
\begin{equation*}
^{J}  W = \{ w \in W \colon\,  D_{L}(w)\subseteq S\setminus J \} =(W^{J})^{-1}, 
\end{equation*}
then every $w \in W$ can be uniquely factorized $
w={_{J} w} \cdot  {^{J}\! w}$, where ${_{J} w} \in W_{J}$, ${^{J} \!
w} \in \,  ^{J} W$,
and $\ell(w)=\ell({_{J} w })+\ell({^{J} \! w})$.

We also need the two following well-known results (a proof of the first can be found in \cite[Lemma~7]{Hom74}, while the second is easy to prove).
\begin{pro}
\label{unicomax}
Let \( J\subseteq S \) and $w \in W$.  The set 
$ W_{J}\cap [e,w] $ has a unique  maximal element  \( w_0(J) \),
so  \( W_{J}\cap [e,w]$ is the interval $[e,w_0(J)] \). 
\end{pro}

We note that the term $w_0(J)$ denotes something different in \cite{BB} and that, if $J=\{s,t\}$, we adopt the lighter notation $w_0(s,t)$ to  mean $w_0(\{s,t\})$.

\begin{pro}
\label{mantiene}
If \( J\subseteq S \) and $v,w \in W$, with $v\leq w$, then $v^J\leq w^J$ and ${^J\! v}\leq \,{^J\!w}$.
\end{pro}

The following elementary results are needed later (see \cite[Lemmas~4.6, 4.7, and~4.8]{CM} for a proof). 
\begin{lem}
\label{lemma0cap}
Fix $H \subseteq S$ and $u=u^H\cdot u_H\in W$. If $j \in D_R(u) \setminus H$, then  $j \in D_R(u^H)$.
\end{lem}

\begin{lem}
\label{lemma00cap}
Fix $u\in W$ and $t,j \in S$, with $t\leq u$ and $j \in D_R(u)$.  If there exists a reduced expression  $X$ for $u$ such that $t \textrm{-}j$ is not a subword of $X$, then $t$ and $j$ commute.
\end{lem}

\begin{lem}
\label{lemma2cap}
Let $u\in W$, $s,t \in S$ with $m_{s,t}\geq 3$, $t\nleq u^{\{s,t\}}$ and $ts\leq u_{\{s,t\}}$. If $j\in  D_{R}(u)\setminus \{s,t\}$, then  $j$ commutes with $s$ and $t$.
\end{lem}

We are interested in the special matchings of a Coxeter group $W$ (to be precise, of intervals in $W$)  partially ordered by  Bruhat order. Given $w\in W$, we say that $M$ is a matching of $w$ if $M$ is a matching of the lower Bruhat interval $[e,w]$.
If \( s\in D_{R}(w) \) (respectively, $s\in D_{L}(w)$), we define a matching 
\(\rho_{s}  \) (respectively, $\lambda_{s}$) of $w$ 
by \( \rho_{s}(u)=us \) (respectively, 
$\lambda_{s}(u)=su$) for all \( u\leq w \). From 
the ``Lifting Property'' (see, for example, \cite[Proposition~2.2.7]{BB} or \cite[Proposition~5.9]{Hum}), it easily follows that $\rho_s$ 
(respectively, $\lambda_s$) is a special matching of $w$. 
We call a matching $M$ of $w$ a \emph{left multiplication
matching} if there exists \( s \in S \) such that  \( M=\lambda _{s} \)
on $[e,w]$, and we call it a  \emph{right multiplication
matching} if there exists \( s \in S \) such that  \( M=\rho _{s} \)
on $[e,w]$.

\begin{defi}
\label{destro}
A  right system for $w\in W$ is a quadruple $\mathcal R=(J,s,t,M_{st})$ such that:
\begin{enumerate}

\item[R1.] $J\subseteq S$, $s\in J$, $t\in S\setminus J$, and $M_{st}$ is a special matching of $w_0(s,t)$ such that  $M_{st}(e)=s$ and  $M_{st}(t)=ts$;

\item[R2.] $(u^{J})^{\{s,t\}}\, \cdot \, M_{st} \Big ((u^{J})_{\{s,t\}} \, \cdot \, _{\{s\}} (u_{J})\Big )\, \cdot \,
 ^{\{s\}}(u_{J}) \leq w$,  for all $u\leq w$;

 \item[R3.] 
if $r\in J$ and $r \leq w^J$, then $r$ and $s$ commute;

\item[R4.]
\label{ddddxxxx}
\begin{enumerate}
\item  if $s\leq (w^J)^{\{s,t\}}$ and  $t\leq (w^J)^{\{s,t\}}$, then $M_{st}= \rho_s$,
\item   if $s\leq (w^J)^{\{s,t\}}$ and  $t\not \leq (w^J)^{\{s,t\}}$, then $M_{st}$ commutes with $\lambda_s$,
\item   if $s\not\leq (w^J)^{\{s,t\}}$ and  $t \leq (w^J)^{\{s,t\}}$, then $M_{st}$ commutes with $\lambda_t$;
\end{enumerate}

\item[R5.] if $s\leq \, ^{\{s\}} (w_{J})$, then $M_{st}$ commutes with $\rho_s$ on $[e,w_0(s,t)]$.

\end{enumerate} 

\end{defi}

\begin{defi}
\label{sinistro}
A left system for $w\in W$ is a quadruple $\mathcal L= (J,s,t,M_{st})$ such that:

\begin{enumerate}
\item[L1.] $J\subseteq S$, $s\in J$, $t\in S\setminus J$, and $M_{st}$  is a special matching of $w_0(s,t)$ such that $M_{st}(e)=s$ and  $M_{st}(t)=st$;

\item[L2.] $    (_Ju)^{\{s\}} \, \cdot \, M_{st} \Big( \, (_Ju)_{\{s\}}  \,  \cdot  \,  _{\{s,t\}} (^J  u) \Big) \, 
\cdot \, ^{\{s,t\}}(^J u)    \leq w$,  for all $u\leq w$;
 \item[L3.] 
if $r\in J$ and $r \leq \, ^Jw$, then $r$ and $s$ commute;

\item[L4.]
\label{puresx}
\begin{enumerate}
\item  if $s\leq \, ^{\{s,t\}}(^Jw)$ and  $t\leq \, ^{\{s,t\}}(^Jw)$, then $M_{st}= \lambda_s$,
\item   if $s\leq \, ^{\{s,t\}}(^Jw)$ and  $t\not \leq \, ^{\{s,t\}}(^Jw)$, then $M_{st}$ commutes with $\rho_s$,
\item   if $s\not\leq \, ^{\{s,t\}}(^Jw)$ and  $t \leq \, ^{\{s,t\}}(^Jw)$, then $M_{st}$ commutes with $\rho_t$;
\end{enumerate}

\item[L5.] if $s\leq \,  (_{J}w)^{\{s\}}$, then $M_{st}$ commutes with $\lambda_s$ on $[e,w_0(s,t)]$.

\end{enumerate} 
\end{defi}
As shown in \cite[Lemma~4.3]{CM}, Properties R5 and L5 are equivalent to the, a priori, more restrictive Properties R5 and L5 appearing in \cite{Mtrans}. We also note that Properties R4 (a) and L4 (a) do not appear in the definitions of right and left systems in \cite{CM} since they appear there in Lemma 4.2 as a consequence of the other properties. Nevertheless we realized that there are a few (irrelevant) exceptions to the statement of \cite[Lemma 4.2]{CM}, so it is more correct to add Properties R4 (a) and L4 (a) in the present definition of right and left systems.

With a right system $\mathcal R=(J,s,t,M_{st})$ for $w$, we associate the matching $M_{\mathcal R}$ on $[e,w]$ sending $u\in [e,w]$ to  
 $$M_\mathcal R (u) = (u^{J})^{\{s,t\}}\, \cdot \, M_{st} \Big( (u^{J})_{\{s,t\}} \,  \cdot  \,  _{\{s\}} (u_{J}) \Big) \, 
\cdot \, ^{\{s\}}(u_{J}).$$
Symmetrically, we associate with any  left system $\mathcal L$ for $w$  the matching $_\mathcal L M$ sending 
$u\in [e, w]$ to 
 $_{\mathcal L}M(u)=\big(M_\mathcal L(u^{-1})\big)^{-1}$, where $M_\mathcal L$ is the matching on $[e,w^{-1}]$ associated to $\mathcal L$ as a right system for $w^{-1}$.

The maps $M_\mathcal R$ and $_{\mathcal L}M$ are actually  matchings, as shown in \cite[Corollary~4.10]{CM}.

The following is the main result of \cite{CM} (see \cite[Theorem~5.9]{CM}).
\begin{thm}
\label{caratteri}
 Let $(W,S)$ be a Coxeter system. If $w\in W$ then
 \begin{enumerate}
 \item the matching associated with a right or left system of $w$ is special;
  \item a special matching of $w$ is the matching associated with a right or left system of $w$;
  \item if $\mathcal R=(J,s,t,M_{st})$ and $\mathcal R'=(J',s',t',M_{s't'})$ are right systems, then $M_\mathcal R=M_{\mathcal R'}$ if and only if $s=s'$, $J\cup C_s=J'\cup C_s$ and one of the following conditions is satisfied:
  \begin{itemize}
   \item $M_{st}(u)=us$ for all $u\leq w_0(s,t)$ and $M_{s't'}(u)=us$ for all $u\leq w_0(s,t')$,
   \item $t=t'$ and $M_{st}=M_{s't'}$;
  \end{itemize}
  \item if $\mathcal R=(J,s,t,M_{st})$ is a right system and $\mathcal L=(K,s',t',M_{s't'})$ is a left system, then $M_\mathcal R=\,_\mathcal L M$ if and only if $s=s'$, $J\cap K \subseteq C_s$, $J\cup K \subseteq S\setminus C_s$, $M_{st}=\rho_s$, $M_{s't'}=\lambda_s$.
\end{enumerate}
\end{thm}

\section{Preliminary results}

We fix an arbitrary Coxeter system $(W,S)$.

For $u,v\in W$, we say that $u$ is a \emph{prefix} of $v$ if $\ell(u)+\ell(u^{-1}v)=\ell(v)$. We similarly define a \emph{suffix}. The proof of the next result is easy and is left to the reader.
\begin{lem}\label{prefix}Let $u,w\in W$ be such that $u$ is a prefix of $w$. If  $u\in W_J$, then $u$ is a prefix of $_Jw$ and $\,^Jw$ is a suffix of $u^{-1}w$. 
\end{lem}

\begin{lem}
\label{ciao}
 Let $w\in W$ and $J,K\subset S$. 
 Then $w^J\in W_K$ if and only if $^Kw\in W_J$, and in this case $w=w^J \cdot v \cdot \,^Kw$ with $\ell(w)=\ell(w^J)+\ell(v)+\ell(\,^Kw)$ and $v\in W_{J\cap K}$.
\end{lem}
\begin{proof}
Assume $w^J\in W_K$. By Lemma~\ref{prefix}, $w^J$ is a prefix of $_Kw$ and $\,^Kw$ is a suffix of $w_J$, and in particular $^Kw\in W_J$. Symmetrically,  Lemma~\ref{prefix} also implies that, if  $^Kw\in W_J$, then $w^J\in W_K$. The second part is now straightforward.
\end{proof}

The reader is invited to recall the definitions in Section \ref{sor}.
Given a system $\mathcal S=(J,H,M)$ for $w$, we always denote $M(e)$ by $s$, the $s$-complement of $J$ by $K$ and, if $|H|=2$, we let $t\in S$ be such that $H=\{s,t\}$.

Our target is to show that special matchings can be described in terms of systems in a very precise way.

The first goal is to identify a class of elements admitting $\mathcal S$-factorizations. This class contains the lower interval $[e,w]$. In order to do so, we separate the proof in two cases. Although this seems unavoidable (see Example \ref{esempio}), the resulting classification is uniform and self-dual.

\begin{defi}
 We say that a system $\mathcal S=(J,H,M)$ for $w$ is of the \emph{first} (respectively, \emph{second}) \emph{kind} if $H\subseteq K$ (respectively, $H\subseteq J$).
\end{defi}
Note that a system $\mathcal S=(J,H,M)$ is of both the first and the second kind if and only if $|H|=1$. 

We now concentrate on systems of the first kind.
Given a system $\mathcal S=(J,H,M)$  of the first kind  for some $w\in W$, we let  
\begin{align*}
 a_\mathcal S(u)&=(u^J)^H \\ b_\mathcal S(u)&=(u^J)_H\cdot\, _H\!(u_J)\\ c_\mathcal S(u)&=\,^H\!(u_J)
\end{align*}
for every $u\in W$, and  
\[
 W_{\mathcal S}=\{u\in W \colon\, u_0(H)\leq w_0(H),\, a_{\mathcal S}(u)\leq a_{\mathcal S}(w)\}.
\]

Observe that $W_\mathcal S$ is an order ideal which contains the lower Bruhat interval $[e,w]$. 

We show next that elements in $W_\mathcal S$ admit $\mathcal S$-factorizations. We need the following result.

\begin{lem}
\label{sseesolose}
 Let $\mathcal S=(J,H,M)$ be a system for $w$ of the first kind. If $u\in W_{S}$, then
 \begin{enumerate}
 \item  $c_{\mathcal S}(u)\geq \;^H(^Ku)$,
 \item $s\leq c_{\mathcal S}(u)$ if and only if $ s\leq \;^H( ^Ku)$.
 \end{enumerate}
\end{lem}
\begin{proof}
 (1). Since $a_\mathcal S(u)b_{\mathcal S}(u) \in W_K$, Lemma~\ref{prefix} implies that $^Ku$ is a suffix of  $c_{\mathcal S}(u)$. Thus   $c_{\mathcal S}(u)\geq \;^Ku = \;^H(^Ku))$ .\\
 (2). The \lq\lq  if part\rq\rq follows by (1). 
 
 Suppose $s\leq c_{\mathcal S}(u)$.  We already know that $^K c_{\mathcal S}(u)= \; ^Ku $ (by Lemma~\ref{mantiene} together with the fact that $^Ku$ is a suffix of  $c_{\mathcal S}(u)$). Hence $c_{\mathcal S}(u)= \;_K(c_{\mathcal S}(u))  \cdot \,^Ku $  and $\,_K(c_{\mathcal S}(u)) \in W_J\cap W_K=W_{C_s}$ since $c_{\mathcal S}(u) \in W_J$. Thus $s\not \leq \,_K(c_{\mathcal S}(u))$  otherwise $s$ would be a left descent of $\,_K(c_{\mathcal S}(u))$ (since all letters in $_K(c_{\mathcal S}(u))$ commute with $s$), so  also a left descent of $c_{\mathcal S}(u)$, which is impossible since $c_{\mathcal S}(u)\in \:^HW$. Hence $s\leq \;^Ku$ and we get the assertion since $^Ku=\;^H(^Ku)$ (by hypothesis, $\mathcal S$ is of the first kind, so $H\subseteq K$). 
\end{proof}

\begin{lem}
\label{see}
 Let $\mathcal S=(J,H,M)$ be a system for $w$ of the first kind. If $u\in W_{S}$, then
\[
 u=a_\mathcal S(u) \cdot b_\mathcal S(u) \cdot c_\mathcal S(u)
\]
is an $\mathcal S$-factorization for $u$.
\end{lem}
\begin{proof}
We show that the 4 properties of the definition of an $\mathcal S$-factorization are satisfied.
\begin{itemize}
\item  The additivity of lengths in the factorization $u=a_\mathcal S(u) \cdot b_\mathcal S(u) \cdot c_\mathcal S(u)$ is clear  by Proposition~\ref{fattorizzo}.
 
  \item Since $a_\mathcal S(u)=(u^J)^H \leq (w^J)^H\leq w^J$ and $w^J \in W_K$ by Property~[S1], we have $a_\mathcal S(u) \in (W_K)^H$, and $a_\mathcal S(u)$ satisfies also the other desired properties since $(w^J)^H$ does by Property~[S2].

  \item Clearly $b_\mathcal S(u) \in W_H$.
  \item Since $c_{\mathcal S}(u)=\,^H\!(u_J)\leq u_J$, we have $c_{\mathcal S}(u)\in \,^H(W_J)$. Since $\mathcal S$ is of the first kind we have $H\cap J=\{s\}$. By Lemma~\ref{sseesolose}, if $s\leq c_{\mathcal S}(u)$ then $s\leq \;^H( \;^Ku)$, and the desired property follows by Property~[S2]. 
 \end{itemize}
\end{proof}

\begin{exa}
\label{esempio}
 We show in this example that the assumption that $\mathcal S$ is of the first kind is necessary for Lemma \ref{see}. We let $J=S=\{r,s,t\}$, $H=\{s,t\}$, $m_{s,r} = 2$, $m_{r,t}\geq 3$,  $m_{s,t}\geq 5$ and $w=rtsts$. We let $M$ be any matching of $[e,tsts]$ which is not a multiplication matching. One can check that  $(J,H,M)$ is a system for $w$ of the second kind. Nevertheless one can also check that $c_{\mathcal S}(w)=rtsts$, so $\textrm{Supp}_H(c_{\mathcal S}(w))=\{s,t\}$. 
\end{exa}

For a system $\mathcal S=(J,H,M)$ of the first kind  and for all $u\in W_{\mathcal S}$, we let 
\[
 M_{\mathcal S}(u)=a_{\mathcal S}(u)M(b_{\mathcal S}(u))c_{\mathcal S}(u).
\]
We observe that the restriction of $M_{\mathcal S}$ to $[e,w_0(H)]$ agrees with $M$.
Our next target is to show that $M_\mathcal S$ is well behaved with respect to the factorization $u=a_\mathcal S(u) \cdot b_\mathcal S(u) \cdot c_\mathcal S(u)$: we need the following preliminary result where, if $A,B\subset W$, we let
\[
A\cdot B=\{ab \colon\, a\in A,\, b\in B\}.
\]

\begin{lem}\label{capraia1}
 Let $\mathcal S=(J,H,M)$ be a system for $w$ of the first kind. If  $u\in W_{\mathcal S}$, then
 \[
  a_\mathcal S (u) \; M(b_\mathcal S (u) )  \in  W^J \cdot W_{s}.
 \]
 where $s=M(e)$.
\end{lem}
\begin{proof}
For short,  we let $a=a_\mathcal S (u)$, $b=b_\mathcal S (u)$, and $c=c_\mathcal S (u)$. By definition,   $ a \cdot b=u^J \cdot  \,  _{s} (u_{J})\in W^J  \cdot W_{s} $. The result holds if $ M( {b}) =  bs $.  This happens, in particular, if  $|H|=1$ (that is, $H=\{s\}$). 
 
 We can therefore assume $M(b)\neq bs$ and $|H|=2$, $H=\{s,t\}$. Note that, in particular, $t\leq M(b)$. We let $\varepsilon \in \{e,s\}$ be such that $s\notin D_R(a \cdot M(b)\varepsilon)$, and we show that $a \cdot M(b)\varepsilon\in W^J$. Otherwise, there exists $j\in J\setminus\{s\}$ such that $j\in D_R(a\cdot M(b) \varepsilon)$.
 We claim that  
\begin{enumerate}
\item $j\in D_R(a) $,
\item $j$ commutes with $s$ and $t$.
\end{enumerate}

Lemma~\ref{lemma0cap} immediately implies that (1) holds. In particular,  $j\leq a=(u^J)^H\leq (w^J)^H\leq w^J\in W_K$, so $j\in K\cap J=C_s$, and hence  $j$ commutes with $s$. We also need to show that $j$ commutes with $t$.
If $t \not \leq a$, we may conclude using Lemma~\ref{lemma00cap} applied to $ aM(b)\varepsilon$. Assume $t  \leq a$, so $M$ commutes with $\lambda_t$ by the definition of $\mathcal S$-factorization. If $M(b)\varepsilon= t$, then
$$e=\lambda_t ( M (b) \varepsilon)= \lambda_t(M(b)) \varepsilon= M (\lambda_t(b)) \varepsilon =
M(tb) \varepsilon,$$ 
so $\varepsilon=M(tb)$, $s\varepsilon= M(\varepsilon)= tb$ and $b= ts\varepsilon= t \varepsilon s= M(b) s$; thus $M(b)=bs$, which contradicts our assumption. Therefore, $M(b)\varepsilon \neq t$; note also that $M(b)\varepsilon \neq e$ since $M(b)\neq bs$. Since $s\notin D_R(M(b)\varepsilon)$, we have $st\leq M(b)\varepsilon$ and $st\neq ts$. Moreover, since $t\leq a$, we have $s\not\leq a$, since $a\cdot b\cdot c$ is an $\mathcal S$-factorization, and we can conclude using Lemma~\ref{lemma2cap} applied to $ aM(b)\varepsilon$.

The claim is proved: so we have that $j\neq s$ satisfies (1) and (2) above. Since $ab\in W^J  \cdot W_{s} $, we have that there exists $\varepsilon' \in W_s$ such that $x=a \cdot b\varepsilon' \in W^J$ (we have $\ell(x)=\ell(a)+\ell(b\varepsilon')$  since $x^H=a$  and $x_H=b\varepsilon'$). But $\ell(xj)=\ell(ajb\varepsilon')= \ell(aj)+\ell(b\varepsilon')=\ell(a)-1+\ell(b\varepsilon')=\ell(x)-1$. In particular $j\in D_R(x)$, which is a contradiction since $x\in W^J$.
\end{proof}

\begin{pro}
\label{weakdefiniscematching}
Let $\mathcal S=(J,H,M)$ be a system for $w$ of the first kind. If $u\in W_\mathcal S$, then

\begin{itemize}
\item  $a_{\mathcal S}(M_\mathcal S(u))=a_{\mathcal S}(u)$;
\item  $b_{\mathcal S}(M_\mathcal S(u))=M(b_{\mathcal S}(u))$;
\item $c_{\mathcal S}(M_\mathcal S(u))=c_{\mathcal S}(u)$.
\end{itemize}
Furthermore, $M_\mathcal S(u)\in W_\mathcal S$.
\end{pro}
\begin{proof}
For short,  we let $a=a_\mathcal S (u)$, $b=b_\mathcal S (u)$, and $c=c_\mathcal S (u)$. 

Let $\varepsilon\in W_{s}$ be such that $a M ( b )\varepsilon  \in W^J$ (see Lemma~\ref{capraia1}). Thus $M_\mathcal S(u)^J=a M(b) \varepsilon$ and $M_\mathcal S(u)_J=\varepsilon c$.  From this parabolic decomposition, it follows that $(M_\mathcal S(u)^J)^{{H}}= a$, $(M_\mathcal S(u)^J)_{{H}}=M ( b )\varepsilon$, $ _{H} (M_\mathcal S(u)_{J})=\varepsilon$ and  $^{H}(M_\mathcal S(u)_{J}) = c$, and the three assertions are proved. 

To show that $M_\mathcal S(u)\in W_\mathcal S$ we only have to verify that $M_{\mathcal S}(u)_0(H)\leq w_0(H)$. This follows from the observation that $u_0(H)=\varepsilon\cdot b_\mathcal S(u)\cdot \varepsilon '$ with $\varepsilon,\varepsilon '\in H\cup \{e\}$.  Indeed we have $M_{\mathcal S}(u)_0(H)\in \{\varepsilon\cdot M(b_\mathcal S(u))\cdot \varepsilon ',M(b_\mathcal S(u))\cdot \varepsilon ',\varepsilon\cdot M(b_\mathcal S(u)), M(b_\mathcal S(u))\}$, so the result follows since in this situation the Bruhat interval $[e,w_0(H)]$ is closed under $M, \lambda_{\varepsilon}$, and $\rho_{\varepsilon'}$.
\end{proof}

\begin{cor}\label{weakdefiniscematching2}
Let $\mathcal S=(J,H,M)$ be a system for $w$ of the first kind. Then $M_\mathcal S$ defines a matching on $W_\mathcal S$. Moreover, for all $u\in W_{\mathcal S}$, we have $u\lhd M_\mathcal S(u)$ if and only if $b_{\mathcal S}(u)\lhd M(b_{\mathcal S}(u))$.
\end{cor}
\begin{proof}
This follows  by Proposition~\ref{weakdefiniscematching}.
\end{proof}

Some comments for systems of the second kind are in order. The situation here is right-to-left symmetrical with respect to systems of the first kind. If $\mathcal S=(J,H,M)$ is a system for $w$ of the second kind, one can define the decomposition
\begin{align*}
 a'_{\mathcal S}(u)&=(_Ku)^H\\
 b'_{\mathcal S}(u)&= (_Ku)^H\cdot ^H(^Ku)\\
 c'_{\mathcal S}(u)&=^H(^Ku)
\end{align*}
for all $u\in W$.   We let 
\[
 W'_{\mathcal S}=\{u\in W\colon\,  u_0(H)\leq w_0(H),\, c'_{\mathcal S}(u)\leq c'_{\mathcal S}(w)\}.
\]
Left versions of Lemmas~\ref{sseesolose}, \ref{see}, and~\ref{capraia1},  Proposition~\ref{weakdefiniscematching}, and  Corollary~\ref{weakdefiniscematching2} hold. In particular, for all $u\in W'_{\mathcal S}$, the factorization $u=a'_{\mathcal S}(u) \cdot b'_{\mathcal S}(u) \cdot c'_{\mathcal S}(u)$ is an $\mathcal S$-factorization. Moreover,  the map
\[
 M_\mathcal S(u)=a'_{\mathcal S}(u) \cdot M(b'_{\mathcal S}(u)) \cdot c'_{\mathcal S}(u)
\]
defines a matching on $W'_{\mathcal S}$.

The problem of consistency of the two definitions of $M_{\mathcal S}$ if $\mathcal S$ is both of the first and second kind disappears after the following fundamental result.

\begin{thm}
\label{fa22}
Let $\mathcal S=(J,H,M)$ be a  system for $w$, and $u\leq w$ (or, more generally, $u\in W_\mathcal S$ if $\mathcal S$ is of the first kind, $u\in W'_\mathcal S$ if $\mathcal S$ is of the second kind). 
 If $u=a\cdot b \cdot c$ is any $\mathcal S$-factorization for $u$,  
 then  
 \[
  M_\mathcal S (u)=a \cdot M(b) \cdot c,
 \]
 (so $a \cdot M(b) \cdot c$ does not depend on the chosen factorization), 
 and $u \lhd M_\mathcal S(u)$ if and only if $b \lhd M(b)$.

 Consequently,  \(
  M_\mathcal S (u)=a \cdot M(b) \cdot c
 \) is an  $\mathcal S$-factorization for  \(
  M_\mathcal S (u) \).
 \end{thm}
\begin{proof}
The main idea of this proof is already present in the proof of \cite[Theorem 5.7]{CM}, and we show here that it is still valid in this more general context with  weaker hypotheses.
We assume that $\mathcal S$ is a system of the first kind, the argument for a system of the second kind being entirely similar.
Let 
\[
 \varepsilon=\begin{cases}
              e& \textrm{if }bs>b;\\ s&\textrm{if } bs<b.
             \end{cases}
\]
We prove the two statements $
  M_\mathcal S (u)=a \cdot M(b) \cdot c
 $
 and $u \lhd M_\mathcal S(u)$ if and only if $b \lhd M(b)$ by induction on $\ell(a)$. We first observe that if $a=a_\mathcal S(u)$, $b=b_\mathcal S(u)$ and $c=c_\mathcal S(u)$, then the results follow by the definition of $M_{\mathcal S}$ and Proposition \ref{weakdefiniscematching}. 
 If $\ell(a)=0$, then $u^J=b\varepsilon$ and $u_J=\varepsilon c$, and therefore $a_{\mathcal S}(u)=e=a$,  $b_{\mathcal S}(u)=b$ and $c_\mathcal S(u)=c$, and we are done.

Assume $\ell(a)\geq 1$. If $ab\varepsilon\in W^J$, that is, $ab\varepsilon =u^J$, then $a_{\mathcal S}(u)=a$, $b_{\mathcal S}(u)=b$ and $c_{\mathcal S}(u)=c$, and the results again follow.

If $ab\varepsilon\notin W^J$, there exists $r\in J\cap K= C_s$ such that $r \in {D}_R(ab \varepsilon)$, and we first claim that $r\in {D}_R(a)$. If $b\varepsilon=e$ this is trivial, otherwise we have  $|H|=2$, $H=\{s,t\}$  and $t\in{D}_R(ab\varepsilon)$. 
Note that $s\notin {D}_R(ab \varepsilon)$ by Proposition \ref{fattorizzo}.
Hence $r\neq s$. By Lemma~\ref{lemma0cap}, we have $r\in {D}_R(a)$. 

Now, if $r$ commutes with both $s$ and $t$, then $a b c=(ar) b (rc)$; this triplet is an $\mathcal S$-factorization and the result clearly follows by induction. 

Therefore, we can assume that $r$ does not commute with $t$. If $b\varepsilon=e$, that is, $b\in\{e,s\}$, then $a b c=(a{r})(b)({r}c)$; we observe that $s\notin {D}_R(a{r})$ since otherwise $s\in {D}_R(a)$, and similarly $s\notin {D}_L({r}c)$. If $t \notin  {D}_R(a {r})$, this triplet is an $\mathcal S$-factorization and the result follows by induction. If $t \in  {D}_R(a {r})$ then $s\not \leq a$ and $M$ commutes with $\lambda_t$ by definition of an $\mathcal S$-factorization,  and again  the   result follows by induction by considering the $\mathcal S$-factorization $(a{r}t)(tb)({r}c)$ since
$$(a{r}t)M(tb)({r}c)= (a{r}t) tM(b)({r}c)= (a{r}) M(b)({r}c)= a M(b)c. $$
The second statement follows since $M(b)\neq tb$ in this case and so $M(b)\lhd b$ if and only if $M(tb)\lhd tb$.

We are therefore reduced to the case $b\varepsilon \neq e$, so $t$ is a right descent of $b\varepsilon$; hence both $t$ and ${r}$ are right descents  of $a b \varepsilon$. In particular, we have a reduced expression for $a b \varepsilon$ which ends with $t\text{-} {r} \text{-}t$ and so $t{r}\leq ab\varepsilon$. This forces $t\leq a$,  so, by definition of an $\mathcal S$-factorization,  
 $M$ commutes with $\lambda_t$ and $s\not \leq a$: in particular, $M(t)=ts$ otherwise $M \circ \lambda_t(e)$ would not be equal to $ \lambda_t \circ M (e)$. 

Now, if $b\varepsilon=t$ we let $m=t{r}t\cdots$ ($m_{t,{r}}$ factors): so $ab\varepsilon=x\cdot m$ with $\ell(ab\varepsilon)=\ell(x)+\ell(m)$ and therefore, since $b\varepsilon=t$, we have $a=x\cdot mt$ with $\ell(a)=\ell(x)+\ell(m)-1$. Moreover, since $r$ and $t$ are both right descents of $xm=ab\varepsilon$, we have that $\ell(xmrt)=\ell(xm)-2=\ell(a)-1$. Thus 
\[
 u=xm \varepsilon c=xmrt\, tr\varepsilon c=(xmrt)(t\varepsilon) (rc)=(xmrt) b (rc).
\]
This factorization is an $\mathcal S$-factorization of $u$ (recall that $s\not \leq a$, so $s\not \leq xmrt$) and hence we can conclude by induction (as we have already observed,  $\ell(xm{r}t)=\ell(a)-1$) that
\begin{align*}
 M_{\mathcal S}(u)&=xm{r}t M(b) {r} c=xm{r}t \, t\varepsilon s\, {r} c\\& =xm \varepsilon s c=a b s c\\&=a M(b) c,
\end{align*}
as, clearly, since $b\varepsilon=t$ we have either $b=ts$ or $b=t$ and in both cases $M(b)=bs$.

We are left with the case $st\leq b\varepsilon$. We observe that $ab \varepsilon$ has a reduced expression which ends with $t\textrm{-}{r}\textrm{-}t$ and a reduced expression which ends with $s\textrm{-}t$; therefore $ab\varepsilon t$ has a reduced expression which ends with $t\textrm{-}r$ and a reduced expression which ends with $s$. In order to transform one of these two reduced expressions to the other using braid moves, we necessarily have to perform a braid relation between $s$ and $t$. Therefore, $tst\cdots$ ($m_{s,t}$ factors) is smaller or equal than $ ab\varepsilon t$.  As we already know  $s\not \leq a$, we deduce $b\varepsilon t \geq sts\cdots$ ($m_{s,t}-1$ factors). But $b\varepsilon t \in W_{H}$ and $t\notin D_R(b\varepsilon t)$ by construction, so $b\varepsilon t=sts\cdots$ ($m_{s,t}-1$ factors). Therefore, $b\varepsilon=sts\cdots$ ($m_{s,t}$ factors). This is a contradiction since $s\notin {D}_R (b\varepsilon)$.

In order to show that \(
  M_\mathcal S (u)=a \cdot M(b) \cdot c
 \) is an  $\mathcal S$-factorization for  \(
  M_\mathcal S (u) \), observe that all axioms of the definition of an $\mathcal S$-factorization follows immediately by the fact that $a\cdot b\cdot c$ is an $\mathcal S$-factorization, except the first one,   $ \ell(M_\mathcal S (u)) =\ell(a) + \ell(M(b)) + \ell(c) $, which  follows  by the fact that    $$u \lhd M_\mathcal S(u) \iff b \lhd M(b).$$
  The proof is complete.
\end{proof}

\section{Main results}
This section contains the main results of this work. In particular, given a system  $\mathcal S$ for $w$, we present a simple way to compute $M_{\mathcal S}(u)$ for every $u\in[e, w]$, and use this to show that $M_{\mathcal S}$ is a special matching. We finally observe, using known results, that all special matchings are of this form and show the desired classification of special matchings.

The next result shows that the action of a  matching $M_\mathcal S$ on an element $u$ can sometimes be calculated also if we do not have an $\mathcal S$-factorization of $u$. The first part is used in this section to prove that $M_{\mathcal S}$ is a special matching, while the second part is needed only in  Section \ref{finalsection}.
\begin{thm}\label{montecristo}Let $\mathcal S=(J,H,M)$ be a system for $w$ and $w=a\cdot b\cdot c$ be an $\mathcal S$-factorization of $w$. Let $u\leq w$ and $u=a' b'c'$ with $a'\leq a$, $b'\leq b$ and $c'\leq c$. Assume that 
\[
 u=a'\cdot b'\cdot c'\,\,\,\textrm{(that is, $\ell(u)=\ell(a)+\ell(b)+\ell(c)$}),
\]
or $M_\mathcal S$ is a special matching of $w$ and there exists $r\in D_L(a)$ with 
\[
 ru=(ra')\cdot b'\cdot c' \,\,\,\textrm{(that is, $\ell(ru)=\ell(ra')+\ell(b')+\ell(c')$}).
\]
Then
\[
 M_\mathcal S(u)=a'M(b')c'.
\]
\end{thm}
\begin{proof}
We first assume that $u=a'\cdot b'\cdot c'$.
 We recall that by definition of $\mathcal S$-factorization we have $|\textrm Supp_H(a)|\leq 1$ and  $|\textrm Supp_H(c)|\leq 1$. We let

 \[
 \eta=\begin{cases}
    e& \textrm{if }H\cap D_R(a')=\emptyset;\\ \alpha &\textrm{if }H\cap D_R(a')=\{\alpha\}.
   \end{cases}                                                                                                                                                                                                                                                                                                                                                                                                                                                                                                                                                                                                               \]
and we let 
\[
 \varepsilon=\begin{cases}
              e& \textrm{if }H\cap D_L(c')=\emptyset;\\ \beta &  \textrm{if } H\cap D_L(c')=\{\beta\}.
             \end{cases}
\]
Now we observe that the factorization
\[
 u=(a' \eta) \cdot (\eta b'\varepsilon) \cdot (\varepsilon c')
\]
is an $\mathcal S$-factorization and that $M$ commutes with the multiplication on the left by $\eta$ (since $\eta \leq a' \leq a$) and the multiplication on the right by $\varepsilon$ (since $\varepsilon \leq c'\leq c$). Therefore, by Theorem~\ref{fa22}
\[
 M_{\mathcal S}(u)=a' \eta  \cdot M(\eta b'\varepsilon) \cdot \varepsilon c'=a'   \eta \cdot \eta M(b')\varepsilon \cdot \varepsilon c'=a' \cdot M(b') \cdot c'.
\]
Now we assume that $ru=(ra')\cdot b'\cdot c'$ and that $M_\mathcal S$ is a special matching of $w$. By the first part, since $ra'\leq a$ by the Lifting Property, we have $M_{\mathcal S}(ru)=ra'M(b')c'$, so if the result fails we have $M_{\mathcal S}(ru)\neq rM_{\mathcal S}(u)$. Applying if necessary $M_{\mathcal S}$ and $\lambda_r$ repeatedly to $u$ we can find an element $x\leq w$ such that $M_{\mathcal S}(x)\lhd x$, $rx\lhd x$ and $M_{\mathcal S}(rx)\neq rM_{\mathcal S}(x)$. We consider an element $x$ of smallest length satisfying these conditions. If $M_{\mathcal S}(x)$ and $rx$ are the only coatoms of $x$, then $x\in W_{\{r,M(e)\}}$, and this is a contradiction since $r\leq a$ and thus $M_{\mathcal S}$ commutes with $\lambda_r$ on $W_{\{r,M(e)\}}$ (by the definition of an $\mathcal S$-factorization if $r\in H$, by the definition of $M_{\mathcal S}$ otherwise). Otherwise there exists an element $y$ such that $y\lhd x$, with $y\notin \{M_{\mathcal S}(x),rx\}$. By the definition of special matchings we have $M_{\mathcal S}(y)\lhd y$ and $ry\lhd y$. By the minimality of $x$ we also have $M_{\mathcal S}(ry)=rM_{\mathcal S}(y)$. Two cases occur and the reader is invited to draw a picture of the following such cases:
\begin{itemize}
 \item $M_{\mathcal S}(y)=ry$: in this case we have $ry\lhd rx$ with $M_{\mathcal S}(ry)\neq rx$ but $M_{\mathcal S}(ry)\not\leq M_{\mathcal S}(rx)$ contradicting the definition of special matching;
 \item $M_{\mathcal S}(y)\neq ry$: in this case we can similarly show that $M_{\mathcal S}(ry)=rM_{\mathcal S}(y)\lhd rM_{\mathcal S}(x)$ with $M_{\mathcal S}(rM_{\mathcal S}(y))\neq rM_{\mathcal S}(x)$ but $M_{\mathcal S}(rM_{\mathcal S}(y))\not\leq M_{\mathcal S}(rM_{\mathcal S}(x))$ contradicting again the definition of special matching.
\end{itemize}
\end{proof}

A right version of the previous result holds (with a similar proof).

Theorem \ref{montecristo} implies the following result.
\begin{cor}\label{yesamatching}
\label{gorgona}
 Let $\mathcal S=(J,H,M)$ be a system for $w$ such that $M_\mathcal S(w)\lhd w$. Then $M_\mathcal S(u)\leq w$ for all $u\in [e, w]$.
\end{cor}
\begin{proof}
Let $w=a\cdot b\cdot c$ and $u=a'\cdot b'\cdot c'$ with $a'\leq a$, $b'\leq b$, and $c'\leq c$. Theorem \ref{montecristo} implies  $M_\mathcal S(u)=a'M(b')c'$. If $M(b') \lhd b'$ then the result is trivial. Otherwise, we have 
\[
 M(b')\leq b
\]
by the Lifting Property for the special matching $M$ (see Lemma~\ref{lpfsm}), since  $b'\leq b$ and $M(b)\lhd b$ (being $M_\mathcal S(w)<w$). We conclude that $a'\cdot M(b') \cdot c'\leq w$ by the subword property.
\end{proof}
The next result tells us that if $H=\{s,t\}$ then we can always assume that $M(t)=ts$ if and only if $t\in K$.

\begin{pro}
\label{elba}
 Let $\mathcal S=(J,H,M)$ be a system for $w$, and $H=\{s,t\}$.  If $t\notin J$ then one of the following applies:
 \begin{itemize}
  \item $M(t)=ts$;
  \item $t\not \leq (w^J)^H$, $\mathcal S'=(J\cup \{t\},H,M)$ is a system for $w$, and $M_{\mathcal S}=M_{\mathcal S'}$.
  \end{itemize}
Symmetrically,  if $t\in J$  then one of the following applies:
 \begin{itemize}
  \item $M(t)=st$;
  \item $t\not \leq \; ^H(^K w)$, $\mathcal S'=(J\setminus \{t\},H,M)$ is a system for $w$, and $M_{\mathcal S}=M_{\mathcal S'}$.
 \end{itemize}
\end{pro}
\begin{proof}
We only prove the first statement. 

 We recall that $s$ and $t$ do not commute. For short, let $J'=J\cup \{t\}$. Clearly $\mathcal S'$ satisfies Property~[S0]   since $\mathcal S$ does.

 If $t\leq (w^J)^H$ then $M$ commutes with $\lambda_t$ and thus $M(t)=M \circ \lambda_t(e)=\lambda_t \circ M(e)=ts$.
 Hence we can assume that $t\not \leq (w^J)^H$. Since
 \[
  w=(w^J)^H \cdot (w^J)_H  \cdot w_J
 \]
 with $(w^J)_H \cdot w_J\in W_{J'}$, we deduce that $w^{J'}\leq (w^J)^H$, so $t \not \leq w^{J'}=(w^{J'})^H$, and if $s\leq (w^{J'})^H$ then $s\leq (w^{J})^H$; thus $\mathcal S'$ satisfies the first part of Property~[S2].
   Moreover, we have $w^{J'}\in W_{K\setminus \{t\}}=W_{K'}$, where $K'$ is the $s$-complement of $J'$; thus $\mathcal S'$ satisfies Property~[S1].
 
In order  to show that $\mathcal S'$ satisfies the second part of Property~[S2], we prove that $t\not \leq \,^H(^{K'}w)$ and that $s \leq  \,^H(^{K'}w)$ implies $s \leq  \,^H(^{K}w)$ (this is enough since $\mathcal S$ satisfies Property~[S2]).

 Since $\mathcal S$ is a system for $w$, Lemma~\ref{ciao} implies 
 \[
  w=w^J \cdot v \cdot \,^Kw=(w^J)^H \cdot  (w^J)_H \cdot  v \cdot  \,^Kw,
 \]
 with $v\in W_{C_s}$. 
 We set, for simplicity, $y=(w^J)_H \cdot v \cdot \,^Kw$, so  $w=(w^J)^H \cdot \,y$.
 Let 
 \[\varepsilon=\begin{cases}
                s&\textrm{if }s\in D_L(v)\\ e&\textrm{otherwise.}
               \end{cases}
\]
Note $t\not \leq \,^Kw$ since $^Kw\in W_J$. The parabolic components of $y$ are as follows: $_Hy= (w^J)_H \cdot \varepsilon$ and $\,^Hy=(\varepsilon v) \cdot \,^Kw$ since $t\not \leq (\varepsilon v) \cdot \,^Kw$ and $s\notin D_L(\,^Kw)$ being $s\in K$ (notice that $ \varepsilon v \in W_{C_s\setminus\{s\}}$).  Finally,  $tst\leq \,_Hy$ since $tst\leq w$ (because $M$ is not a multiplication matching on $[e,w_0(H)]$) and $t$ does not appear in the other factors of $w=(w^J)^H \cdot \,_Hy \cdot \,^Hy$.

Therefore, we have that the factorization $w=(w^J)^H \cdot y$ is a factorization of the form $w=x\cdot a$ which satisfies the following properties
\begin{itemize}
 \item $\ell(w)=\ell(x)+\ell(a)$,
 \item $x\in W_{K'}$,
 \item $\,^Ha \in W_J$,
 \item $s\leq \,^Ha$ if and only if $s \leq \,^Kw$.
\end{itemize}
We claim that if $w=x\cdot a$ is a factorization satisfying these properties with $\ell(x)$ as big as possible, then $x=\,_{K'}w$ and $a=\,^{K'}w$. From the claim, it follows that $t\not \leq \,^H(\,^{K'}w)=\,^Ha$ and $s \leq \,^H(\,^{K'}w)$ if and only if $s \leq \,^H(^Kw)$, since $\,^H(^Kw)=\,^Kw$ (because $H\subseteq K$).

Let  $x\cdot a$ satisfy the conditions above with $\ell(x)$ as big as possible: we have to show that $a$ does not have left descents in $K'$. By contradiction,  let  $r\in K' \cap D_L(a)$. 

If  $r=s$, then $s \in D_L(\,_Ha)$ and the factorization $w=(xs)\cdot (sa)$ still satisfies the conditions above (since $_H(sa)=s(_Ha)$ and $\,^H(sa)=\,^Ha$) and contradicts the maximality of $\ell(x)$.

If  $r\neq s$ and  $s\notin D_L(a)$, then $_Ha$ has a prefix $t\text{-}s\text{-}t$ and by (the left version of) Lemma~\ref{lemma2cap} we deduce that $r$ commutes with $s$ and $t$. Therefore, we have
\[
 w=(xr) \cdot (_Ha) \cdot (r\,^Ha).
\]
We observe that  $r\in D_L(\,^Ha)$, by  Lemma~\ref{lemma0cap},  and that $r\,^Ha\in \,^HW$, since $r$ commutes with $s$ and $t$. In particular the factorization
\[
 w=(xr)\cdot (ra)
\]
still contradicts the maximality of $\ell(x)$ since $_H(ra)=\,_Ha$ and $^H(ra)=r\,^Ha$.

We need to show that $M_{\mathcal S}=M_{\mathcal S'}$. Let $u\leq w$: since $\mathcal S$ is a system for $w$, Lemma~\ref{ciao} implies 
\[
  u=u^J \cdot \bar{v} \cdot \,^Ku=(u^J)^H \cdot  (u^J)_H \cdot  \bar{v} \cdot  \,^Ku
 \]
with $\bar{v}\in W_{C_s}$. Let 
 \[\bar{\varepsilon}=\begin{cases}
                s&\textrm{if }s\in D_L(\bar{v})\\ e&\textrm{otherwise.}
               \end{cases}
\]
Then  the factorization
\[
 u= (u^J)^H \cdot  \Big((u^J)_H \cdot \bar{\varepsilon} \Big) \cdot (\bar{\varepsilon} \bar{v} \cdot  \,^Ku)
\]
is both an $\mathcal S$-factorization and an $\mathcal S'$-factorization: the unique non immediate property follows by the fact that $t\not\leq \bar{\varepsilon} \bar{v} \cdot  \,^Ku$ and the fact that $s\leq \bar{\varepsilon} \bar{v} \cdot  \,^Ku$ implies $s \leq \, ^Ku= \,^H(^Ku)\leq \,^H(^Kw)$.  
Thus $M_{\mathcal S}(u)=M_{\mathcal S'}(u)$ by Theorem~\ref{fa22}.
\end{proof}

\begin{cor}
\label{giannutri}
 Let $w\in W$. If $M$ is a special matching of $w$, then there exists a system $\mathcal S$ for $w$ such that $M=M_{\mathcal S}$.
 Vice versa, if $\mathcal S$ is a system for $w$ such that $M_\mathcal S(w)\lhd w$, then $M_\mathcal S$ is a special matching of $w$.
\end{cor}
\begin{proof}
 We use Theorem~\ref{caratteri}. Let $M$ be a special matching of $w$ and suppose that $M$ is associated with a  right system $\mathcal R=(J,s,t,M)$ for $w$.
 If $M$ is a multiplication matching on $[e,w_0(s,t)]$, then $\mathcal S=(J,\{s\},M)$ is a system for $w$ and $M_\mathcal R=M_\mathcal S$.
 If $M$ is not a multiplication matching on $[e,w_0(s,t)]$, then $\mathcal S=(J,\{s,t\},M)$ is a system and $M_\mathcal R=M_\mathcal S$. The situation is similar in the case when $M$ is associated with a left systems.
 
Now let $\mathcal S=(J,H,M)$ be a system for $w$ such that $M_\mathcal S(w)\lhd w$. If $|H|=1$ (that is, $H=\{s\}$), then $M_\mathcal S = M_\mathcal R$ where $\mathcal R=(J,s,t,\rho_s)$ and $t$ is any element in $(S\setminus J)\cup (C_s\setminus \{s\})$.
 If $|H|=2$, $H=\{s,t\}$, then it follows by Proposition~\ref{elba} that $M_\mathcal S = M_\mathcal R$ with 
 $$\mathcal R=\begin{cases}
 \text{ the right system $(J,s,t, M)$}                      & \text{if $t\notin J$ and $M(t)=ts$}\\
\text{ the right system $(J\setminus \{t\},s,t,M)$} & \text{if $t\in J$ and $M(t)=ts$}\\
\text{ the left system $(K\setminus\{t\},s,t,M)$}   & \text{if $t\notin J$ and $M(t)=st$}\\
\text{ the left system $(K,s,t,M)$}                         & \text{if $t\in J$ and $M(t)=st$}
 \end{cases}$$
where $K$ is the $s$-complement of $J$. 
\end{proof}
For $w\in W$, let $SM_w$ be the set of systems $\mathcal S=(J,H,M)$ for $w$ such that $M_\mathcal S(w)\lhd w$ and such that, for all $\alpha \in H$, $M_{\mathcal S}(\alpha)=s\alpha$ if and only if $\alpha \in J$.

\begin{thm}
 The set 
 \[
 \{ M_{\mathcal S}: \mathcal S\in SM_w\}
 \]
is a complete list of all distinct special matchings of $w$.
\end{thm}
\begin{proof}
Straightforward by Theorem~\ref{caratteri} and Corollary~\ref{giannutri}.
\end{proof}
\begin{rem}We observe that one of the main achievements in Definition \ref{sistema} is the elimination of Properties R2 and L2 that we have in Definitions \ref{destro} and \ref{sinistro}: we no longer need to know in advance that the associated map is a matching on the lower Bruhat interval (which was also the most difficult and technical part to verify in the definition of right and left systems): it is now a consequence of the other axioms (Corollary \ref{yesamatching}).  
 
\end{rem}

\section{The combinatorial invariance}\label{finalsection}

In this section we show how the main result of \cite{BCM1} (that is, the combinatorial invariance of Kazhdan-Lusztig polynomials of lower Bruhat intervals  in any arbitrary Coxeter group) can be easily deduced from results in the previous sections.

We fix a system $\mathcal S=(J,H,M)$ for $w\in W$ and we simply denote $M_{\mathcal S}$ by $M$.
\begin{pro}\label{comm}
Suppose that $[e,w]$ is not a dihedral interval. There exists a multiplication matching $N$ of $w$ such that $M(w)\neq N(w)$ and $MN(u)=NM(u)$ for all $u\leq w$.
\end{pro}
\begin{proof}
 Let $w=a\cdot b\cdot c$ be an $\mathcal S$-factorization. If $a=c=e$ then $[e,w]$ is a dihedral interval contradicting our hypothesis. We may therefore assume $a\neq e$, the case $c\neq e$ being entirely similar. Let $r\in D_L(a)$ and we prove the result with $N=\lambda_r$. It is clear that $\lambda_r(w)\neq M(w)$ since they are obtained from a reduced expression of $w$ (namely a concatenation of reduced expressions of $a$, $b$ and $c$) by deleting two distinct letters.

 Now let $u\leq w$ and  $u= a' \cdot b' \cdot c'$ with $a'\leq a$, $b'\leq b$ and $c'\leq c$.

 The first part of Theorem \ref{montecristo} implies $M(u)=a'M(b')c'$. Moreover, the second part of Theorem \ref{montecristo} implies $M(ru)=ra'M(b')c'$ since $ra' \leq a$ (by the Lifting Property) and $r(ru)=(r(ra'))\cdot b'\cdot c'$.  The result follows.
\end{proof}

Let $M$ be a special matching of $w$ and $N$ be a multiplication matching of $w$ such that  $M(w)\neq N(w)$ and $MN(u)=NM(u)$ for all $u\leq w$ ($N$ exists by  Proposition~\ref{comm}). Fix $u\in [e,w]$. We let $\mathcal O_u$ be the orbit of $u$ under the action of the group generated by the  special matchings $M$ and $N$, so
$$\mathcal O_u= \begin{cases} 
\{u,M(u)\}  & \text{ if $M(u)=N(u)$} \\ 
\{u,M(u), N(u),MN(u)\} & \text{ if $M(u)\neq N(u)$.}
\end{cases}
$$
We let $\mathcal U$ be the free $\mathbb Z[q]$-module with basis $\mathcal O_u$.
\begin{lem}
\label{Fcomm} Suppose $|\mathcal O_u|=4$ and let, for all $x\in \mathcal O_u$, the polynomials $f_x,\tilde{f}_x,g_x,\tilde{g}_x\in \mathbb Z[q]$ be such that $f_x=f_{N(x)}$, $\tilde{f}_x=\tilde{f}_{N(x)}$, $g_x=g_{M(x)}$ and $\tilde{g}_x=\tilde{g}_{N(x)}$. Consider
 the endomorphisms $F,G$ of $\mathcal U$ uniquely determined by
\[
 F(x)=f_xx+\tilde{f}_xM(x)\qquad G(x)=g_xx+\tilde{g}_xN(x)
\]
for all $x\in \mathcal O_u$. Then $F \circ G=G \circ F$.

\end{lem}
\begin{proof}
Given  $x\in \mathcal O_u$, we have
 \begin{align*}
  F\circ G(x)&=F((g_xx+\tilde{g}_xN(x))\\
  &= f_xg_x x+f_{N(x)}\tilde{g}_xN(x)+\tilde{f}_xg_xM(x)+\tilde{f}_{N(x)}\tilde{g}_xMN(x).
 \end{align*}
and  similarly 
 \[
  G\circ F(x)= g_xf_x x+g_{M(x)}\tilde{f}_xM(x)+\tilde{g}_xf_xN(x)+\tilde{g}_{M(x)}\tilde{f}_xNM(x).
 \]

The result follows. 
\end{proof}
The combinatorial invariance of Kazhdan-Lusztig polynomials is equivalent to the analogous result for $R$-polynomials (see \cite[Chapter 5]{BB}). $R$-polynomials are polynomials in $\mathbb Z[q]$ indexed by pairs of elements of $W$ and it is well known that they can be computed with a recursive algorithm based on the following facts. If $N$ is a multiplication matching of $w$, then
\[
 R_{u,w}=(q^{c(N,u)}-⁠1)R_{u,N(w)}+q^{c(N,u)}R_{N(u),N(w)},
\]
where \[
 c(N,u)=\begin{cases}
         1&\textrm{if }N(u)\rhd u\\0& \textrm{if }N(u)\lhd u,
        \end{cases}
\]
for all $u<w$. Moreover, $R_{u,u}=1$ for all $u\in W$, and $R_{u,v}=0$ if $u\not\leq v$.

The next result, which was originally proved in \cite{BCM1},  shows that one can substitute $N$ in the previous formula with any special matching $M$, and that in particular $R$-polynomials can be computed from the knowledge of the lower Bruhat interval as an abstract poset.
\begin{thm}
Let  $M$ be a special matching of $w$. If $u\in [e, w]$, then
\[
 R_{u,w}= (q^{c(M,u)}-⁠1)R_{u,M(w)}+q^{c(M,u)}R_{M(u),M(w)},
\]
where
\[
 c(M,u)=\begin{cases}
         1&\textrm{if }M(u)\rhd u\\0& \textrm{if }M(u)\lhd u.
        \end{cases}
\]
\end{thm}
\begin{proof}
 The result is known if $M$ is a multiplication matching and is easy to prove if $[e,w]$ is a dihedral interval. Otherwise, let $N$ be a multiplication matching of $w$ such that  $M(w)\neq N(w)$ and $MN(u)=NM(u)$ for all $u\leq w$, which exists by  Proposition~\ref{comm}. We consider the endomorphisms $F$ and $G$ of $\mathcal U$ such that, for $x\in \mathcal O_u$,
\[
 F(x)=(q^{c(M,x)}-⁠1)x+q^{c(M,x)}M(x)
\]
and
\[
 G(x)=(q^{c(N,x)}-⁠1)x+q^{c(N,x)}N(x).
\]
If $|\mathcal O_u|=2$ then $F(x)=G(x)$, and if  $|\mathcal O_u|=4$ then $FG=GF$ by Lemma \ref{Fcomm}. In particular, in both cases, we have $F\circ G=G\circ F$.

 For all $z\in W$, we consider the $\mathbb Z[q]$-morphism $v_z \colon\, \mathcal U\rightarrow \mathbb Z[q]$ given by
 \[
  v_z(x)=R_{x,z}
 \]
 for all $x\in \mathcal O_u$ and extended by linearity. The statement is equivalent to
 \begin{equation}\label{main}
  v_w(x)=v_{M(w)}(F(x))
 \end{equation}
 with $x=u$. We prove Eq. \eqref{main} for all $x\in \mathcal U$  by induction on $\ell(w)$. We have
 \begin{align*}
  v_w(x)&=v_{N(w)}(G(x))\\
  &=v_{MN(w)}(FG(x))
 \end{align*}
 where we use in the first line that Eq. \eqref{main} holds for $N$ and $G$ since $N$ is a multiplication matching, and in the second line our induction hypothesis, since $M$ restricts to a special matching of $N(w)$. Now, since $M\circ N=N \circ M$ and $F\circ G=G\circ F$, we can conclude 
 \begin{align*}
  v_w(x)&=v_{NM(w)}(GF(x)\\
  &=v_{M(w)}(F(x))
 \end{align*}
 where we  use again Eq. \eqref{main} for $N$ and $G$.
\end{proof}

\end{document}